\documentclass[a4paper,11pt]{amsart}

\usepackage{amssymb,amsmath,amsfonts,amsthm,mathabx,mathrsfs,pdfsync,
enumerate,bbm,fancyhdr,enumerate,graphicx, url,afterpage,setspace,geometry,mathabx}

\usepackage[dvipsnames]{xcolor}
\usepackage[colorlinks=true, pdfstartview=FitV, linkcolor=blue, citecolor=blue, urlcolor=blue]{hyperref}

\usepackage[english]{babel}
\usepackage[utf8]{inputenc}

\newtheorem{theorem}{Theorem}[section]
\newtheorem{lemma}[theorem]{Lemma}
\newtheorem{corollary}[theorem]{Corollary}
\newtheorem{proposition}[theorem]{Proposition}

\newtheorem{definition}[theorem]{Definition}

\numberwithin{equation}{section}

\DeclareMathOperator{\conv}{conv}
\DeclareMathOperator{\spn}{span}

\DeclareMathOperator{\id}{id}
\DeclareMathOperator{\C}{\mathbb{C}}

\DeclareMathOperator{\N}{\mathbb{N}}
\DeclareMathOperator{\h}{\mathcal{H}}

\DeclareMathOperator{\B}{\mathcal{B}}

\begin{document}

\title{$M_d$-multipliers of a locally compact group}
\author{Bat-Od Battseren}
\maketitle
\begin{abstract} We show that the space $M_d(G)$ of $M_d$-multipliers of a locally compact group $G$ is isometrically isomorphic to the Banach space of bounded functionals on the $d$-fold Haagerup tensor product of $L^1(G)$ vanishing on the kernel of the convolution map. Consequently, we see that $M_d(G)$ is isometrically isomorphic to the dual space of $X_d(G)$, the completion of $L^1(G)$ in the dual of $M_d(G)$. We also show that $M_d$-type-approximation-properties are inherited to lattices.
\end{abstract}

\section*{INTRODUCTION}
In \cite{Pis05}, a commutative Banach algebra $M_d(\Gamma)$ was introduced in pursuit of studying Dixmier's Similarity Problem, where $\Gamma$ is a discrete group and $d$ is an integer greater than one. When $d=2$, this is the algebra of Herz-Schur multipliers of $\Gamma$. In general, $M_d(\Gamma)$ stands between the Fourier-Stieltjes algebra $B(\Gamma)$ and the algebra of bounded functions $\ell^\infty(\Gamma)$. By the usual $(\ell^\infty(\Gamma),\ell^1(\Gamma))$-duality, we can embed $\ell^1(\Gamma)$ in the dual of $M_d(\Gamma)$. Denote by $X_d(\Gamma)$ the completion of $\ell^1(\Gamma)$ in $M_d(\Gamma)^*$. Pisier showed in \cite{Pis05} that $M_d(\Gamma)\cong X_d(\Gamma)^*$ isometrically using Paulsen-Smith's theorem \cite{PS87}. In essence, it means the space $M_d(\Gamma)$ can be alternatively described as the dual of a quotient of the $d$-fold Haagerup tensor product of  $\ell^1(\Gamma)$. 

This last construction can be performed for a locally compact group $G$ by using $L^1(G)$ instead of $\ell^1(\Gamma)$. However, as indicated in \cite[Remark after Theorem 2.10]{Pis05}, it is not clear how the functionals in this space would correspond to (continuous) functions on $G$. In the discrete case, this correspondence was clear because $\Gamma$ is naturally contained in $\ell^1(\Gamma)$. We elaborate on this problem for locally compact second countable groups.

In \cite{Bat23}, the author defined the space $M_d(G)$ of \textit{$M_d$-multipliers} of a locally compact group $G$ as the subspace of all continuous functions in $M_d(G_\mathrm{d})$, where $G_\mathrm{d}$ denotes the discrete realization of $G$. More precisely:
\begin{definition}\label{defn Md}
Let $G$ be a locally compact group. A bounded continuous function $\varphi$ on $G$ is called an \textit{$M_d$-multiplier} if there exist Hilbert spaces $\h_0,\dots, \h_d$ and bounded maps $\xi_i:G\rightarrow \B(\h_i,\h_{i-1})$ for $i=1,\dots,d$ such that $\h_0=\h_d=\C$ and 
\begin{align}\label{eq eq}
\varphi(x_1\cdots x_d) = \xi_1(x_1)\cdots \xi_d(x_d) (1) \quad \text{for all}\quad x_1,\dots,x_d\in G.
\end{align}
We denote by $M_d(G)$ the space of all $M_d$-multipliers and endow it with the norm
\begin{align*}
\|\varphi\|_{M_d} = \inf \prod_{i=1}^d \|\xi_i\|_\infty
\end{align*} where the infimum is taken over all $\h_i$'s and $\xi_i$'s satisfying \eqref{eq eq}.

\end{definition}
This definition is compatible with the definition given by Pisier for discrete groups. Because the identity map $M_d(G_\mathrm{d})\rightarrow \ell^\infty (G_\mathrm{d})$ is contracting, $M_d(G)= M_d(G_\mathrm{d})\cap C(G)$ becomes a commutative Banach algebra. Our main result is the following, which demonstrates that the two previously mentioned constructions of $M_d(G)$ are the same for locally compact second countable groups. 
\begin{theorem}\label{thm1}
Let $G$ be a locally compact second countable group and let $d$ be an integer greater than one. 
\begin{enumerate}
\item If $\varphi$ is an $M_d$-multiplier of $G$, i.e. $\varphi\in M_d(G) = M_d(G_\mathrm{d})\cap C(G)$, then there is a  bounded functional $\hat{\varphi}$ on $L^1(G)\otimes_h \cdots \otimes_h L^1(G)$ such that $\|\hat{\varphi}\|=\|\varphi\|_{M_d}$ and
\begin{align*}
\hat{\varphi} (f_1\otimes \cdots \otimes f_d) = (\varphi, f_1*\cdots *f_d)_{(L^\infty(G),L^1(G))}
\end{align*}
for all $f_1,\dots,f_d\in L^1(G)$.
\item Conversely, if $\phi$ is a bounded functional  on $L^1(G)\otimes_h \cdots \otimes_h L^1(G)$ vanishing on the kernel of the convolution map
$
\conv_d:L^1(G)\otimes \cdots \otimes L^1(G)\rightarrow L^1(G),
$ then there is an $M_d$-multiplier $\check{\phi}\in M_d(G)$ such that $\|\check{\phi}\|_{M_d}=\|\phi\|$ and
\begin{align*}
(\check{\phi}, f_1*\cdots *f_d)_{(L^\infty(G),L^1(G))} = \phi(f_1\otimes \cdots \otimes f_d)
\end{align*} for all $f_1,\dots,f_d\in L^1(G)$.
\end{enumerate}
Moreover, we have $\varphi = \check{\hat{\varphi}}$ and $\phi=\hat{\check{\phi}}$.
\end{theorem}
If we restrict the $(L^\infty(G),L^1(G))$-duality on $M_d(G)$, we can embed $L^1(G)$ in $M_d(G)^*$. Denote by $X_d(G)$ the completion of $L^1(G)$ in $M_d(G)^*$.  After Theorem \ref{thm1} and Paulsen-Smith's theorem, it is easily seen that $X_d(G)$ is a Banach predual of $M_d(G)$.
\begin{theorem}\label{thm2}
Let $G$ be a locally compact second countable group and let $d$ be an integer greater than one. Then we have $M_d(G)\cong X_d(G)^*$ isometrically.
\end{theorem}
This duality between $M_d(G)$ and $X_d(G)$ was used in \cite{Ver23,Bat23b} to define $M_d$ type approximation properties. We recall the definitions.
\begin{definition}[$M_d$ type approximation properties]\label{defn Md AP} Let $G$ be a locally compact group and let $d$ be an integer greater than one.

We say that $G$ has \textit{$M_d$-approximation-property} ($M_d$-AP in short) if  there is a net $(\varphi_i)$ of functions in $C_c(G)\cap M_d(G)$ such that $\varphi_i\rightarrow  1$ in $\sigma (M_d(G),X_d(G))$-topology. 

We say that $G$ is \textit{$M_d$-weakly-amenable} ($M_d$-WA in short) if  there is a net $(\varphi_i)$ of functions in $C_c(G)\cap M_d(G)$ such that $\sup_i \|\varphi_i\|_{M_d} \leq C<\infty$ and $\varphi_i\rightarrow 1$ in $\sigma (M_d(G),X_d(G))$-topology. The number $\Lambda_{CH}(G,d) = \inf C$ is called the \textit{$M_d$-Cowling-Haagerup constant}. We conventionally write $\Lambda_{CH}(G,d) = \infty$ if $G$ is not $M_d$-WA. 

We say that $G$ has \textit{$M_d$-weak-Haagerup-property} ($M_d$-WH in short) if  there is a net $(\varphi_i)$ of functions in $C_0(G)\cap M_d(G)$ such that
$\sup_i \|\varphi_i\|_{M_d} \leq C<\infty$ and $\varphi_i\rightarrow 1$ in $\sigma (M_d(G),X_d(G))$-topology. The number $\Lambda_{WH}(G,d) = \inf C$ is called the \textit{$M_d$-weak-Haagerup constant}. We conventionally write $\Lambda_{WH}(G,d) = \infty$ if $G$ does not have $M_d$-WH.

\end{definition}
When $d=2$, these are exactly the usual approximation property, weak amenability, and weak Haagerup property \cite{DCH85,CH89,HK94,Knu16,Haa16}. Allow us to mention a couple of motivations to study $M_d$ type approximation properties in general extent. The first one is related to the Similarity Problem: A group with $M_d$-AP for all $d\in \N$ has an affirmative answer to the Similarity Problem \cite[Theorem 1.2]{Ver23}. The second one is their stability: $M_d$ type approximation properties are stable under von Neumann equivalence \cite{Bat23b}, hence also under Measure equivalence and W*-equivalence. See \cite{IPR19,Ish21,Bat23,Bat23b} for more about von Neumann equivalence.

The following result allows one to relax the convergence conditions in the definition of $M_d$-WA and $M_d$-WH.
\begin{proposition}\label{prop uniform conv}
Let $G$ be a locally compact second countable group and let $d$ be an integer greater than one. Then $G$ is $M_d$-WA if and only if there is a net $(\varphi_i)$ of functions in $C_c(G)\cap M_d(G)$ such that $\sup_i\|\varphi_i\|_{M_d}\leq C< \infty$ and $\varphi_i\rightarrow  1$ uniformly on compact subsets. We have $\inf C = \Lambda_{CH}(G,d)$. Similarly, $G$ has $M_d$-WH  if and only if there is a net $(\varphi_i)$ of functions in $C_0(G)\cap M_d(G)$ such that $\sup_i \|\varphi_i\|_{M_d} \leq C'<\infty$ and $\varphi_i\rightarrow 1$ uniformly on compact subsets. We have $\inf C' = \Lambda_{WH}(G,d)$.
\end{proposition}
A key step in proving Theorem \ref{thm1} is Lemma \ref{lem cont} in which we observe that we can choose the maps $\xi_i:G\rightarrow \B(\h_i,\h_{i-1})$ in Definition \ref{defn Md} to be continuous where $\B(\h_i,\h_{i-1})$ is endowed with the weak operator topology. Proposition \ref{prop uniform conv} is a simple consequence of this observation. Since the uniform convergence on compact subsets pass to closed subgroups via the restriction map, the following is immediate.
\begin{corollary}\label{cor consts}
If  $G$ is a locally compact second countable group and $H$ is a closed subgroup of $G$, then we have $\Lambda_{CH}(H,d)\leq \Lambda_{CH}(G,d)$ and $\Lambda_{WH}(H,d) \leq \Lambda_{WH}(G,d)$ for all integers $d\geq 2$.
\end{corollary}
As shown in \cite[Theorem 3]{Bat23}, the converse inequalities hold true when $H$ is a lattice of $G$, hence the following.
\begin{theorem}
If $G$ is a locally compact second countable group and $\Gamma$ is a lattice of $G$, then we have $\Lambda_{CH}(\Gamma,d)=\Lambda_{CH}(G,d)$ and $\Lambda_{WH}(\Gamma,d)=\Lambda_{WH}(G,d)$ for all integers $d\geq 2$.
\end{theorem}
The argument used to prove Corollary \ref{cor consts} does not work for $M_d$-AP due to the lack of uniform norm boundedness. However, if we assume $H$ is a lattice in $G$, there is an easier approach to obtain a similar results.

\begin{theorem}\label{thm AP}
If $G$ is a locally compact second countable group and $\Gamma$ is a lattice of $G$, then $G$ has  $M_d$-AP if and only if $\Gamma$ has $M_d$-AP.
\end{theorem}

We acknowledge that this work is strongly inspired by \cite{Haa16} and \cite{Pis05} where the case $d=2$ and the discrete case are studied.
\section{Preliminary}
Throughout the paper, $G$ is a locally compact second countable group, $G_\mathrm{d}$ is the discrete realization of $G$, $d\in \N$ is an integer greater than one, and  $\h_i$, $\h$, and $\mathcal{K}$ are Hilbert spaces. We conventionally write $\h_0 = \h_d = \C$ and $v_0 = v_d = 1$.

\subsection{Fourier and Fourier-Stieltjes algebras}
Recall that a \textit{unitary representation} of $G$ is a continuous group homomorphism $\pi:G\rightarrow \mathcal{U}(\h)$, where $\h$ is a Hilbert space and $\mathcal{U}(\h)$ is the group of all unitary operators on $\h$ endowed with the weak operator topology (WOT). Given vectors $v,w\in\h$, we call the function 
\begin{align}\label{coef}
\varphi = \pi_{v,w}: g\in G\mapsto \langle\pi(x)v,w\rangle
\end{align} the \textit{coefficient} of $\pi$ associated to $v,w\in \h$. These are typical elements of $M_d(G)$. To see that,
put  $\h_1 =\dots =\h_{d-1} = \h$,   $\xi_1 (x)= \langle \pi(x)\cdot , w\rangle$, $\xi_d(x): 1\mapsto \pi(x)v$, and $\xi_2(x)=\dots= \xi_{d-1}(x) = \pi(x)$. Then we have
\begin{align*}
\xi_1 (x_1)\cdots \xi_d(x_d)(1) = \langle \pi(x_1\cdots x_d)v,w\rangle = \pi_{v,w}(x_1\cdots x_d)
\end{align*} with $\|\pi_{v,w}\|_{M_d}\leq \|v\|\|w\|$.

The \textit{Fourier-Stieltjes algebra} $B(G)$ consist of all coefficients of unitary representations of $G$. Endowed with the  pointwise product and the norm $\|\varphi\|_B = \inf \|v\|\|w\|$, where the infimum is taken over all $(\h,\pi,v,w)$ satisfying \eqref{coef}, the Fourier-Stieltjes algebra is a commutative Banach algebra.

The \textit{Fourier algebra} $A(G)$ is the smallest closed subalgebra of $B(G)$ containing $B(G)\cap C_c(G)$. Some basic facts about the Fourier algebra include that it consists of the coefficients of the left regular representation $\lambda:G\rightarrow \mathcal{U}(L^2(G))$ and that if $\varphi\in A(G)$ then we can find $f,g\in L^2(G)$ such that
$\varphi = \lambda_{f,g}$ and $\|\varphi\|_B = \|f\|_2\|g\|_2$. To distinguish that $\varphi$ is in the Fourier algebra, we write $\|\varphi\|_A = \|\varphi\|_B$. We refer the readers to \cite{Eym64,KL18,Run20} for more about these algebras.

The Fourier algebra contains abundant continuous functions so that when a function $\varphi$ multiplies the Fourier algebra, i.e. $\varphi A(G)\subseteq A(G)$, it is necessarily continuous. We will make use of this fact. The following lemma is comparable to the well known fact \cite[Theorem V.7.3]{Gaa73} that every weakly measurable unitary representation on a separable Hilbert space is continuous.
\begin{lemma}\label{lem meas cont}
Let  $A,B:G\rightarrow \h$ be two bounded weakly measurable maps (i.e. $\h$ is endowed with $\sigma(\h,\h^*)$-Borel space structure) into a separable Hilbert space $\h$. Assume that $\varphi: G\rightarrow \C$ is a function such that
$\varphi (y x^{-1}) = \langle A(x),B(y)\rangle$ for all $x,y\in G$. Then $\varphi$ multiplies the Fourier algebra $A(G)$. In particular, $\varphi$ is continuous.
\end{lemma}

\begin{proof}
Let $(e_n)_{n\in \N}$ be an orthonormal basis for $\h$. By assumption, there exist uniformly bounded measurable functions $a_n,b_n\in L^\infty (G)$ such that
$A(x) = \sum_n a_n(x) e_n$, $B(y) = \sum_n b_n(y) e_n$, and
\begin{align*}
\varphi(yx^{-1}) = \langle A(x),B(y)\rangle= \sum_n a_n (x)\overline{b_n(y)}.
\end{align*}
Take any $f,g\in L^2(G)$. By Cauchy–Schwarz inequality, we have 
\begin{align*}
\sum_n \|a_nf\|_2 \|b_n g\|_2 &\leq 
(\sum_n \|a_nf\|_2^2) ^{1/2}(\sum_n  \|b_n g\|_2^2)^{1/2} 
\\
& = (\int_G |f(x)|^2 \sum_n |a_n(x)|^2 dx)^{1/2} (\int_G |g(x)|^2 \sum_n |b_n(x)|^2 dx)^{1/2}
\\
&\leq \|A\|_\infty \|B\|_\infty \|f\|_2\|g\|_2 <\infty,
\end{align*} thus $\sum_n\lambda_{a_nf,b_ng}\in A(G)$. On the other hand, if we fix $x\in G$,
since the bounding function
\begin{align*}
|f(x^{-1}y)\overline{g(y)} \sum_{n=1}^N a_n (x^{-1}y)\overline{b_n(y)}| \leq 
|f(x^{-1}y)\overline{g(y)}|  \|A\|_\infty\|B\|_\infty 
\end{align*} is integrable over $y\in G$,
by the dominated convergence theorem we have
\begin{align*}
(\varphi\lambda_{f,g})(x) &= \int_G \varphi(x) f(x^{-1}y)\overline{g(y)}dy =  \int_G \langle A(x^{-1}y),B(y) \rangle f(x^{-1}y)\overline{g(y)}dy 
\\
&= \int_G \sum_n a_n (x^{-1}y)\overline{b_n(y)}f(x^{-1}y)\overline{g(y)}dy = (\sum_n\lambda_{a_nf,b_ng})(x).
\end{align*}
In other words, $\varphi$ is a Fourier multiplier, hence continuous.
\end{proof} 
A similar statement can be achieved for any $M_d$-multipliers. More precisely, if $\varphi\in M_d(G_\mathrm{d})$ and the $\xi_i$'s in \eqref{eq eq} are weakly measurable, then we can deduce that $\varphi\in M_d(G)$. This will be used  to see that $\check{\phi}$ in Theorem \ref{thm1} is continuous.

\subsection{The space $M_d(G)$ of $M_d$-multipliers}
The key step in proving Theorem \ref{thm1}.(i) is to show that we can choose $\xi_i$'s in \eqref{eq eq} to be WOT-continuous, i.e. the maps $x\in G \mapsto \langle \xi_i(x)v_i,v_{i-1} \rangle$ are continuous for any $v_i\in \h_i$ and $v_{i-1}\in \h_{i-1}$. The case $d=2$ is done in \cite[Theorem 3.2]{Haa16}. We adapt it to the general case. Beforehand, let us denote
\begin{align*}
A_i = \{\xi_{i+1}(x_{i+1})\cdots \xi_d (x_d)v_d \in \h_i\mid x_{i+1},\dots ,x_d\in G\},\quad i=0,\dots,d-1,
\end{align*} and 
\begin{align*}
B_i = \{[\xi_1 (x_1)\cdots \xi_i (x_i)]^* v_0 \in\h_i\mid x_1,\dots ,x_{i-1}\in G\},\quad i=1,\dots,d.
\end{align*} 

\begin{lemma}\label{lem total}
If $\varphi\in M_d(G)$ is non-zero, we can choose $\h_i$'s and $\xi_i$'s in \eqref{eq eq} such that $A_i$ and $B_i$ are total in $\h_i$ for all $i=1,\dots, d-1$.
\end{lemma}
\begin{proof} Let $P_i\in \B(\h_i)$ be the projection onto the closed linear span of $A_i$. Note that $P_0$ is just the identity map because $\varphi$ is nonzero and $\h_0 \cong \C$. It is easily checked that the Hilbert spaces $P_i\h_i$ and the bounded maps $P_{i-1}\xi_i$ also satisfy \eqref{eq eq}. Thus replacing $(\h_i,\xi_i)$'s by $(P_i\h_i,P_{i-1}\xi_i)$'s, we can and will assume that each $A_i$ is total in $\h_i$. 

Let $Q_i \in\B(\h_i)$ be the projection onto the closed linear span of $B_i$. Again, $Q_d$ is just the identity map on $\h_d$ because firstly
\begin{align*}
\varphi(x_1\cdots x_d) = \langle v_d, [\xi_1 (x_1)\cdots \xi_d (x_d)]^* v_0\rangle
\end{align*} is nonzero for some $x_1,\dots , x_d \in G$ and secondly $\h_d\cong\C$. Put $\h_i' = Q_i\h_i$ and $\xi_i' (x_i) = Q_{i-1} \xi_i(x_i)$  for all $i=1,\dots ,d$ with $Q_0 = \id_{\h_0}$. Then $(\h_i',\xi_i')$'s still satisfy \eqref{eq eq}. To see that, take any vectors $v_i' = [\xi_1(x_1)\cdots \xi_i(x_i)]^*v_0 \in\h_i'$ and $v_{i+1}\in \h_{i+1}'$. Then we have
\begin{align*}
\langle Q_i\xi_{i+1}(x_{i+1})Q_{i+1} v_{i+1}',v_i'\rangle &= 
\langle v_{i+1}',Q_{i+1} [\xi_1(x_1)\cdots \xi_i(x_{i+1})]^*v_i'\rangle
\\ 
&=\langle v_{i+1}',[\xi_1(x_1)\cdots \xi_i(x_{i+1})]^*v_i'\rangle
\\
&=\langle Q_i\xi_{i+1}(x_{i+1})v_{i+1}',v_i'\rangle.
\end{align*} Since such $v_i'$'s are total in $\h_i'$, the above equality is true for any $v_i'\in \h_i'$. In other words, we have 
\begin{align}\label{eq remove Q}
Q_i\xi_{i+1}(x_{i+1}) Q_{i+1} = Q_i\xi_{i+1}(x_{i+1}).
\end{align} From this, it is easily seen that
\begin{align*}
 \xi_1'(x_1)\cdots \xi_d'(x_d)  v_d&= 
 Q_0\xi_1(x_1)Q_1\xi_2(x_2) Q_2\cdots Q_{d-1}\xi_d(x_d)  v_d
 \\
 &= 
 \xi_1(x_1)\cdots \xi_d(x_d) v_d = \varphi(x_1\cdots x_d).
\end{align*}
Moreover, the set
\begin{align*}
B_i'= \{[\xi_1' (x_1)\cdots \xi_i' (x_i)]^* v_0 \in\h_i\mid x_1,\dots ,x_{i-1}\in G\}
\end{align*} is total in $\h_i'$. Now, the main concern is whether 
\begin{align*}
A_i' = \{\xi_{i+1}'(x_{i+1})\cdots \xi_d' (x_d)v_d \in \h_i\mid x_{i+1},\dots ,x_d\in G\}
\end{align*} is total in $\h_i'$. Let us prove it by reverse recursion. Firstly,
\begin{align*}
A_{d-1}' = \{Q_{d-1}\xi_d (x_d)v_d \in \h_i\mid x_d\in G\}
\end{align*} is total in $\h_{d-1}' = Q_{d-1} \h_{i-1}$ because $A_{d-1}$ is total in $\h_{d-1}$ by assumption. Assume that $A_{i+1}'$ is total in $\h_{i+1}' = Q_{i+1} \h_{i+1}$. Then by \eqref{eq remove Q} and the totality condition of $A_i$'s, we have
\begin{align*}
\overline{\spn} A_i' 
& = \overline{\spn} \bigcup_{x_{i+1}\in G} \xi_{i+1}' (x_{i+1})A_{i+1}' v_d = \overline{\spn} \bigcup_{x_{i+1}\in G} Q_i \xi_{i+1} (x_{i+1})Q_{i+1}\h_{i+1}
\\
& = \overline{\spn} \bigcup_{x_{i+1}\in G} Q_i \xi_{i+1} (x_{i+1})\h_{i+1} = \overline{\spn} Q_i A_i = Q_i \h_i = \h_i', 
\end{align*} showing that $A_i'$ is total in $\h_i'$. This proves the lemma.
\end{proof}
\begin{lemma}\label{lem cont}
Let $\varphi\in M_d(G)$ be non-zero and let $\h_i$'s and $\xi_i$'s be as in Lemma \ref{lem total}. Then $\xi_i$'s are WOT-continuous and $\h_i$'s are separable.
\end{lemma}
\begin{proof}
Choose $\h_i$'s and $\xi_i$'s as in Lemma \ref{lem total}. Suppose that $x_i^{(j)}\rightarrow x_i$ in $G$. Take any vectors 
\begin{equation*}
v_i =\xi_{i+1} (x_{i+1})\cdots \xi_d(x_d)  v_d \in \h_i,\quad v_{i-1} =[\xi_1(x_1) \cdots \xi_{i-1}(x_{i-1)}]^*v_0 \in \h_{i-1}.
\end{equation*} Then by continuity of $\varphi$ we have
\begin{align*}
\langle \xi_i (x_i^{(j)}) v_i, v_{i-1}\rangle &=\varphi(x_1 \cdots x_i^{(j)}\cdots x_d) \rightarrow \varphi(x_1\cdots x_d) = \langle \xi_i (x_i) v_i, v_{i-1}\rangle.
\end{align*} By uniform boundedness of $\xi_i$ and totality assumptions, the above convergence is true for any $v_i
\in \h_i$ and $v_{i-1}\in \h_{i-1}$. This shows that each $\xi_i$ is WOT-continuous. Now, $\h_i$ is separable since $G$ is second countable, $\xi_i$  is WOT-continuous, and $A_i$ is dense in $\h_i$.
\end{proof}
\begin{lemma}\label{lem conv action}
If $f\in L^1(G)$ and $\varphi\in M_d(G)$, then we have $f*\varphi\in M_d(G)$ with $\|f*\varphi\|_{M_d}\leq \|f\|_1 \|\varphi\|_{M_d}$.
\end{lemma}
\begin{proof} The continuity of $f*\varphi$ is a well known fact of harmonic analysis. Let $\h_i$'s and $\xi_i$'s be as in Lemma \ref{lem cont}. Define $\xi_1':G\rightarrow \B(\h_1,\h_0)$ as
\begin{align*}
\xi_1'(x)v_1 = \int_G f(y) (\xi_1(y^{-1}x)v_1) dy
\end{align*} for all $x\in G$ and $v\in \h_1$. Note that $\|\xi_1'\|\leq \|f\|_1\|\xi_1\|$ and 
\begin{align*}
\xi_1'(x_1)\xi_2(x_2)\cdots \xi_d(x_d)v_d &= \int_G f(y) (\xi_1(y^{-1}x_2)\xi_2(x_2)\cdots \xi_d(x_d)v_d) dy
\\
&=\int_G f(y)\varphi (y^{-1}x_1\cdots x_d) dy 
\\
&= (f*\varphi)(x_1\cdots x_d)
\end{align*} for all $x_1,\dots,x_d\in G$. Therefore, $f*\varphi\in M_d(G)$ with $\|f*\varphi\|_{M_d}\leq \|f\|_1\|\varphi\|_{M_d}$.
\end{proof}

\begin{proof}[Proof of Proposition \ref{prop uniform conv}]
Let $(\varphi_i)$ be a net in $M_d(G)$ such that $\varphi_i\rightarrow 1$ in $\sigma(M_d(G),X_d(G))$-topology. It follows that $\varphi_i\rightarrow 1$ in $\sigma(L^\infty (G),L^1(G))$-topology. By uniform boundedness principle, we have $\sup_i\|\varphi\|_\infty<\infty$. 
Take any compactly supported positive function $f\in C_c(G)$ with $\|f\|_1=1$ and put $\varphi_i' = f*\varphi_i$. 
Now the proof of \cite[Lemma 2.2]{Haa16} works too see that $\varphi_i'\rightarrow 1$ uniformly on compact subsets. We note that if $\varphi_i$ is compactly supported, then so is  $\varphi_i'$. Similarly, if $\varphi_i$ vanish at infinity, then so does $\varphi_i'$. By Lemma \ref{lem conv action}, we have $\varphi_i'\in M_d(G)$ with $\|\varphi_i'\|_{M_d}\leq \|\varphi_i\|_{M_d}$. This shows that if we replace the $\sigma(M_d(G),X_d(G))$-convergence by uniform convergence on compact subsets in the definition of $M_d$-WA and $M_d$-WH, we will get weaker properties. Conversely, it is always true that uniform convergence on compact subsets implies $\sigma(L^\infty (G),L^1(G))$-convergence. Moreover, for a uniformly bounded net in $M_d(G)$, its $\sigma(L^\infty (G),L^1(G))$-convergence implies $\sigma(M_d (G),X_d(G))$-convergence as $L^1(G)\subseteq X_d(G)$ is dense. This completes the proof.
\end{proof}

\subsection{The dual space $M_d'(G)$ of $X_d'(G) = L^1(G)\otimes_h\cdots \otimes_h L^1(G)$}
Let us recall the following theorem by Paulsen-Smith (see \cite[Theorem 3.2]{PS87}, \cite[Theorem 9.4.4]{ER22}, and \cite[Corollary 5.4]{Pis03}).
\begin{theorem}\label{thm ER}
Given operator spaces $V_1$, \dots, $V_d$, a linear mapping 
\begin{align*}
\phi: V_1\otimes \dots \otimes V_d\rightarrow \B(\h_d,\h_0)
\end{align*}
extends to a completely bounded map on the Haagerup tensor product if and only if there exist Hilbert spaces $\h_1$, \dots, $\h_{d-1}$ and completely bounded linear maps 
\begin{align*}
\eta_i : V_i \rightarrow \B(\h_i,\h_{i-1}),\quad i=1,\dots, d
\end{align*} such that 
\begin{align}\label{eq eq1}
\phi (v_1\otimes \cdots \otimes v_d) = \eta_1(v_1)\cdots \eta_d(v_d)
\end{align} for all $v_1\in V_1$, \dots, $v_d\in V_d$. Moreover, we can assume that $\|\phi\|_{cb} = \|\eta_1\|_{cb}\cdots \|\eta_d\|_{cb}$.
\end{theorem}
Conventionally, we only consider $\h_0=\h_d=\C$. Denote
\begin{align*}
X_d ' (G)=\underbrace{L^1(G)\otimes_h\cdots \otimes_h L^1(G)}_{d}
\end{align*} where $L^1(G)$ is endowed with its maximal operator space structure. Denote by $M_d' (G)$ the (operator) dual of $X_d'(G)$. If we put $V_i = L^1(G)$ and $\h_0=\h_d=\C$ in Theorem \ref{thm ER}, we simply get $\|\phi\|_{cb} = \|\phi\|$ because $\C\cong \B(\C,\C)$ is abelian \cite[Proposition 2.2.6]{ER22},  and $\|\eta_i\|_{cb} =\|\eta_i\|$ because $L^1(G)$ has its maximal operator space structure \cite[Section 3.3]{ER22}. With this simplification, we get the following version of Theorem \ref{thm ER}.
\begin{corollary}\label{cor ER}
A linear map $\phi: L^1(G)\otimes \cdots\otimes L^1(G)\rightarrow \C$ extends to a bounded linear functional $\phi\in M_d'(G) = X_d'(G)^*$ on $X_d' (G)$ if and only if 
there exist Hilbert spaces $\h_i$'s and bounded linear maps $\eta_i:L^1(G)\rightarrow \B(\h_i,\h_{i-1})$ such that
\begin{align}\label{eta}
\phi (f_1\otimes \cdots \otimes f_d) = \eta_1(f_1)\cdots \eta_d(f_d)
\end{align}  for all $f_1,\dots,f_d\in L^1(G)$. 
Moreover, we can assume that $\|\phi\| = \|\eta_1\|\cdots \|\eta_d\|$.
\end{corollary} 

Let $\phi\in M_d'(G)$ and let $\h_i$'s and $\eta_i$'s be as in \eqref{eta}. Denote 
\begin{align*}
C_i = \{\eta_{i+1}(f_{i+1})\cdots \eta_d (f_d)v_d \in \h_i\mid f_{i+1},\dots ,f_d\in L^1(G)\},\quad i=0,\dots,d-1,
\end{align*} and 
\begin{align*}
D_i = \{[\eta_1 (f_1)\cdots \eta_i (f_i)]^* v_0 \in\h_i\mid f_1,\dots ,f_{i-1}\in L^1(G)\},\quad i=1,\dots,d.
\end{align*} 
We have an analogue of Lemma \ref{lem total}. The proof is done in exactly the same way by replacing $(A_i,B_i,\xi_i, x_i)$'s by $(C_i,D_i,\eta_i,f_i)$'s. We present the proof for convenience.
\begin{lemma}\label{lem total2}
If $\phi\in M_d'(G)$ is non-zero, we can choose $\h_i$'s and $\eta_i$'s such that $C_i$ and $D_i$ are total in $\h_i$. In this case, $\h_i$'s are separable.
\end{lemma}
\begin{proof} Let $P_i\in \B(\h_i)$ be the projection onto $\overline{\spn} C_i$. $P_0$ is the identity because $\phi$ is nonzero and $\h_0 \cong \C$. The Hilbert spaces $P_i\h_i$ and the bounded maps $P_{i-1}\eta_i$ also satisfy \eqref{eta}, thus we can and will assume that each $C_i$ is total in $\h_i$. Let $Q_i \in\B(\h_i)$ be the projection onto $\overline{\spn} D_i$. Again, $Q_d$ is just the identity on $\h_d$ because 
\begin{align*}
\phi(f_1\otimes\cdots \otimes f_d) = \langle v_d, [\eta_1 (f_1)\cdots \eta_d (f_d)]^* v_0\rangle
\end{align*} is nonzero for some $f_1,\dots , f_d \in L^1(G)$ and $\h_d\cong\C$. Put $\h_i' = Q_i\h_i$ and $\eta_i' (f_i) = Q_{i-1} \eta_i(f_i)$  for all $i=1,\dots ,d$ with $Q_0 = \id_{\h_0}$. Then $(\h_i',\eta_i')$'s still satisfy \eqref{eta}. Indeed, take any vectors $v_i' = [\eta_1(f_1)\cdots \eta_i(f_i)]^*v_0 \in\h_i'$ and $v_{i+1}\in \h_{i+1}'$. Then we have
\begin{align*}
\langle Q_i\eta_{i+1}(f_{i+1})Q_{i+1} v_{i+1}',v_i'\rangle &= 
\langle v_{i+1}',Q_{i+1} [\eta_1(f_1)\cdots \eta_i(f_{i+1})]^*v_i'\rangle
\\ 
&=\langle v_{i+1}',[\eta_1(f_1)\cdots \eta_i(f_{i+1})]^*v_i'\rangle
\\
&=\langle Q_i\eta_{i+1}(f_{i+1})v_{i+1}',v_i'\rangle.
\end{align*} Since such $v_i'$'s are total in $\h_i'$, the above equality is true for any $v_i'\in \h_i'$. In other words, we have 
\begin{align}\label{eq remove QQ}
Q_i\eta_{i+1}(f_{i+1}) Q_{i+1} = Q_i\eta_{i+1}(f_{i+1}).
\end{align} From this, it is easily seen that
\begin{align*}
 \eta_1'(f_1)\cdots \eta_d'(f_d)  &= 
 Q_0\eta_1(f_1)Q_1\eta_2(f_2) Q_2\cdots Q_{d-1}\eta_d(f_d)  
 \\
 &= 
 \eta_1(f_1)\cdots \eta_d(f_d)= \phi(f_1\otimes \cdots \otimes f_d).
\end{align*}
Moreover, the set
\begin{align*}
D_i'= \{[\eta_1' (f_1)\cdots \eta_i' (f_i)]^* v_0 \in\h_i\mid f_1,\dots ,f_{i-1}\in L^1(G)\}
\end{align*} is total in $\h_i'$. Let us prove that
\begin{align*}
C_i' = \{\eta_{i+1}'(f_{i+1})\cdots \eta_d' (f_d)v_d \in \h_i\mid f_{i+1},\dots ,f_d\in L^1(G)\}
\end{align*} is total in $\h_i'$ by reverse recursion. Firstly,
\begin{align*}
C_{d-1}' = \{Q_{d-1}\eta_d (f_d)v_d \in \h_i\mid f_d\in L^1(G)\}
\end{align*} is total in $\h_{d-1}' = Q_{d-1} \h_{d-1}$ because $C_{d-1}$ is total in $\h_{d-1}$. Assume that $C_{i+1}'$ is total in $\h_{i+1}' = Q_{i+1} \h_{i+1}$. Then by \eqref{eq remove QQ} and the totality condition of $C_i$'s, we have
\begin{align*}
\overline{\spn} C_i' 
& = \overline{\spn} \bigcup_{f_{i+1}\in L^1(G)} \eta_{i+1}' (f_{i+1})C_{i+1}' = \overline{\spn} \bigcup_{f_{i+1}\in L^1(G)} Q_i \eta_{i+1} (f_{i+1})Q_{i+1}\h_{i+1}
\\
& = \overline{\spn} \bigcup_{f_{i+1}\in L^1(G)} Q_i \eta_{i+1} (f_{i+1})\h_{i+1} = \overline{\spn} Q_i C_i = Q_i \h_i = \h_i', 
\end{align*} showing that $C_i'$ is total in $\h_i'$. This proves the first assertion of the lemma. Recall that $G$ is second countable, equivalently $L^1(G)$ is separable. Now, by density and continuity, it follows that the Hilbert spaces $\h_i'$ are separable.
\end{proof}
Consider the $d$-convolution map
\begin{align}\label{convd}
\begin{aligned}
\conv_d : \underbrace{L^1(G)\otimes\cdots \otimes L^1(G)}_d &\rightarrow  L^1(G)
\\
f_1\otimes \cdots \otimes f_d &\mapsto f_1*\cdots *f_d.
\end{aligned}
\end{align} 
Denote by $Y$ the closure of the kernel of $\conv_d$ in $X_d'(G)$. Denote by $Y^\perp \subseteq M_d' (G)$ the closed subspace of functionals that vanish on $Y$ (equivalently on the kernel of $\conv_d$). We note that a functional $\phi \in M_d'(G)$ is in $Y^\perp$ if and only if there is a map  $
\tilde{\phi}:L^1(G)  \rightarrow \C$ such that
\begin{align}\label{tildephi}
\tilde{\phi} (f_1*\cdots *f_d) = \phi(f_1\otimes \cdots \otimes f_d)v_d.
\end{align} Theorem \ref{thm1} suggests that $M_d(G)$  and $Y^\perp$ are in a natural correspondence. Then Theorem \ref{thm2} is straightforward since $Y^\perp \cong (X_d' (G)/Y)^*$ isometrically.

\section{Main section}
In this section, we will prove the main results.
\subsection{Mapping from $M_d(G)$ into $Y^\perp\subseteq M_d' (G)$}
For $\varphi\in M_d(G)$, we write
\begin{align}\label{hat}
\hat{\varphi} (f_1\otimes \cdots \otimes f_d) = (\varphi,f_1*\cdots*f_d)_{(L^\infty(G),L^1(G))}
\end{align} for all $f_1,\dots,f_d\in L^1(G)$. Since $\conv_d$ is linear, $\hat{\varphi}$ extends to a linear functional on $L^1(G)\otimes \cdots\otimes L^1(G)$ and vanish on the kernel of $\conv_d$.
\begin{lemma}\label{lem Md to Y} The map $\varphi \in M_d(G)\mapsto \hat{\varphi}\in Y^\perp\subseteq M_d'(G)$ is a well defined contraction.
\end{lemma}
\begin{proof} 
If $\xi:G\rightarrow \B(\h,\mathcal{K})$ is a bounded weakly measurable map, we can extend it to a bounded linear map  $\hat{\xi}:L^1(G)\rightarrow \B(\h,\mathcal{K})$ as
\begin{align*}
\langle\hat{\xi}(f)v,w\rangle = \int_G f(x) \langle\xi(x)v,w\rangle dx,\quad v\in\h,\quad w\in\mathcal{K}.
\end{align*} Clearly, we have $\|\hat{\xi}\| \leq \sup_{x\in G}\|\xi(x)\|$. For $\varphi\in M_d(G)$, we choose separable $\h_i$'s and  WOT-continuous $\xi_i$'s as in Lemma \ref{lem cont}.
Observe that
\begin{align}\label{eq hat varphi}
\begin{aligned}
\langle &\hat{\xi}_1(f_1)\cdots \hat{\xi}_d(f_d)v_d,v_0\rangle =
\int_G f_1(x_1)
\langle \xi_1(x_1)\hat{\xi}_2(f_2)\cdots \hat{\xi}_d(f_d)v_d,v_0\rangle dx_1
\\
&=
\int_G\int_G f_1(x_1)f_2(x_2)
\langle \xi_2(x_2)\hat{\xi}_3(x_3)\cdots \hat{\xi}_d(f_d)v_d,\xi_1(x_1)^*v_0\rangle dx_2 dx_1
\\ &\dots
\\
&=\int_G\cdots \int_G f_1(x_1)\cdots f_d(x_d)
\langle \xi_d(x_d)v_d,\xi_{d-1}(x_{d-1})^*\cdots \xi_1(x_1)^*v_0\rangle dx_{d}\cdots dx_1
\\
&=\int_G\cdots \int_G f_1(x_1)\cdots f_d(x_d)
 \varphi(x_1\cdots x_d) dx_d\cdots dx_1
\\
&= (\varphi, f_1*\cdots *f_d)_{(L^\infty(G),L^1(G))} = \hat{\varphi} (f_1\otimes \cdots \otimes f_d).
\end{aligned}
\end{align} 
for all $f_1,\dots,f_d\in L^1(G)$. By Corollary \ref{cor ER}, $\hat{\varphi}$ extend to a bounded linear functional on $X_d'(G)$ with 
\begin{align*}
\|\hat{\varphi}\| \leq \prod_{i=1}^d \|\hat{\xi}_i\|\leq \prod_{i=1}^d \|{\xi}_{i}\|_\infty= \|\varphi\|_{M_d}.
\end{align*} This proves the lemma.
\end{proof}
\subsection{Mapping from $Y^\perp\subseteq M_d' (G)$ into $M_d(G)$}
We want to define a contraction
\begin{align*}
\phi\in Y^\perp \mapsto \check{\phi}\in M_d(G)
\end{align*} that is inverse to the map in Lemma \ref{lem Md to Y}. Let $(k_n)_{n\in\N}$ be a bounded left-right approximate identity of $L^1(G)$ consisting of compactly supported, positive, self-adjoint functions, i.e. $k_n\in C_c(G)_+$, $k_n^* = k_n$, $\|k_n\|_1 \leq 1$, $\|k_n * f-f\|_1\rightarrow 0$, and $\|f*k_n-f\|_1\rightarrow 0$ for all $f\in L^1(G)$. For $\phi\in Y^\perp$, define $\check{\phi}$ as 
\begin{align}\label{check phi}
\check{\phi}: x\in G\mapsto  \lim_n \phi (\lambda(x) k_n\otimes \cdots \otimes k_n)v_d.
\end{align} We will see soon that the above is a well defined continuous function on $G$.
\begin{lemma}\label{lem Y to Md}
The map $\phi\in Y^\perp\subseteq M_d'(G) \mapsto \check{\phi}\in M_d(G)$ is a well defined contraction.
\end{lemma}
\begin{proof}
Let $\phi\in Y^\perp$. Choose $\h_i$'s and $\eta_i$'s as in Lemma \ref{lem total2}. Denote $h_n = k_n*\cdots *k_n$. Then $(h_n)_{n\in\N}$ is also a bounded approximate identity with the same properties. Take any $x\in G$. For $v_i =\eta_{i+1}(f_i * f_{i+1}) \eta_{i+2}(f_{i+2})\cdots \eta_d(f_d) v_d \in \h_i$ and $v_{i-1} = [\eta_1(f_1) \cdots \eta_{i-1}(f_{i-1})]^*v_0 \in \h_{i-1}$ with $f_1,\dots,f_d\in L^1(G)$, define the operator $\check{\eta}_i(x)$ as 
\begin{align}\label{eq3}
\begin{aligned}
&\langle \check{\eta}_i(x) v_i,v_{i-1}\rangle = \langle \phi(f_1\otimes \cdots \otimes f_{i-1} \otimes \lambda(x) f_i\otimes f_{i+1}\otimes \cdots \otimes f_d)v_d,v_0\rangle
\\
&= \langle \eta_1(f_1) \cdots \eta_{i-1}(f_{i-1}) \eta_i(\lambda(x) f_i)\eta_{i+1}(f_{i+1}) \cdots \eta_d(f_d) v_d,v_0\rangle
\\
&= \lim_n \langle \eta_1(f_1)\cdots \eta_{i-1}(f_{i-1}) \eta_i(\lambda(x) h_n * f_i)\eta_{i+1}(f_{i+1}) \cdots \eta_d(f_d) v_d,v_0\rangle
\\
&= \lim_n \langle \tilde{\phi}(f_1*\cdots *f_{i-1}*\lambda(x) h_n * f_i *\cdots *f_d) ,v_0\rangle
\\
&= \lim_n \langle \eta_1(f_1)\cdots \eta_{i-1}(f_{i-1}) \eta_i(\lambda(x) h_n)\eta_{i+1}(f_i*f_{i+1}) \cdots \eta_d(f_d) v_d,v_0\rangle
\\
&=\lim_n \langle \eta_i(\lambda(x) h_n) {v_{i}},{v_{i-1}}\rangle.
\end{aligned}
\end{align} This extends to a well defined linear map on the linear span of $v_i$'s because $\phi\in Y^\perp$ and $\h_0 = \h_d =\C$.  Since $L^1(G)*L^1(G)$ is dense in $L^1(G)$ and $\eta_{i+1}$ is continuous, such $v_i$'s are total in $\h_i$. By density and boundedness, $\check{\eta}_i (x)$ extends to a bounded linear map in $\B(\h_{i},\h_{i-1})$ with norm not exceeding $\|\eta_i\|$. In other words, we have $\check{\eta}_i: G\rightarrow \B(\h_i,\h_{i-1})$  well defined and uniformly bounded by $\|\eta_i\|$. Observe that
\begin{align*}
& \check{\eta}_1(x_1)\cdots \check{\eta}_d (x_d)v_d = \lim_{n_d} \cdots \lim_{n_1} \eta_1 (\lambda (x_1) h_{n_1})\cdots \eta_d(\lambda (x_d) h_{n_d})v_d 
\\ 
&= \lim_{n_d} \cdots \lim_{n_1} \tilde{\phi} (\lambda (x_1) h_{n_1} * \cdots* \lambda (x_{d-1}) h_{n_{d-1}}* \lambda (x_d) k_{n_d}*\cdots *k_{n_d})
\\
&= \lim_{n_d} \cdots \lim_{n_1}  \eta_1(\lambda (x_1) h_{n_1} * \cdots* \lambda (x_{d-1}) h_{n_{d-1}}* \lambda (x_d) k_{n_d}) \eta_2(k_{n_d})\cdots \eta_d(k_{n_d}) v_d 
\\
&= \lim_{n_d} \eta_1(\lambda (x_1\cdots x_d) k_{n_d}) \eta_2(k_{n_d})\cdots \eta_d(k_{n_d}) v_d
\\
&= \lim_{n} \phi(\lambda (x_1\cdots x_d) k_{n} \otimes k_{n}\otimes \cdots \otimes k_{n}) v_d.
\end{align*} 
This shows that the function $\check{\phi}$ from \eqref{check phi} is well defined and in $M_d(G_\mathrm{d})$ with
\begin{align*}
\|\check{\phi}\|_{M_d} \leq \prod_{i=1}^d \|\check{\eta}_i\|\leq   \prod_{i=1}^d \|\eta_i\| = \|\phi\|.
\end{align*}
Since the map $
x\in G\mapsto \lambda(x) g\in L^1(G)$ is continuous for every $g\in C_c(G)$ (see for example \cite[Lemma 6.6.11]{Ped89}), $\langle\check{\eta}_i (\cdot)v_i,v_{i-1}\rangle$ is a pointwise limit of continuous functions, hence measurable. 
Now, the continuity of $\check{\phi}$ follows from Lemma \ref{lem meas cont} by putting 
\begin{align*}
\h= \h_{d-1},\quad A(x) = \check{\eta}_d(x^{-1})v_d, \quad B(y) = [\check{\eta}_1(y))\check{\eta}_2(e)\cdots \check{\eta}_{d-1}(e)]^*v_0.
\end{align*}
This completes the proof.
\end{proof}
\begin{proof}[Proof of Theorem \ref{thm1}]
It only remains to prove that the maps in Lemma \ref{lem Md to Y} and Lemma \ref{lem Y to Md} are inverse to each other. If $\varphi\in M_d(G)$, by definition \eqref{hat} and \eqref{check phi} we have 
\begin{align*}
\check{\hat{\varphi}} (x) &= \lim_n \hat{\varphi} (\lambda(x)k_n\otimes \cdots \otimes k_n) v_d 
\\
&= \lim_n (\varphi, \lambda(x) k_n*\cdots *k_n)_{(L^\infty(G),L^1(G))} 
\\
& = \varphi (x).
\end{align*} for almost every $x\in G$, but since both $\check{\hat{\varphi}}$ and $\varphi$ are continuous, they coincide. Conversely, if $\phi\in Y^\perp$, we have
\begin{align*}
\hat{\check{\phi}} &(f_1\otimes \cdots \otimes f_d) = (\check{\phi}, f_1*\cdots *f_d)_{(L^\infty(G),L^1(G))}
\\
&= \int_{G^n} \check{\phi}(x_1\cdots x_d) f_1(x_1)\cdots f_d(x_d) dx_1 \cdots dx_d
\\
&= \int_{G^n}  \langle\check{\eta}_1(x_1)\cdots \check{\eta}_d (x_d)v_d ,v_0 \rangle f_1(x_1)\cdots f_d(x_d) dx_1 \cdots dx_d
\\
&=\langle{\eta}_1(f_1)\cdots{\eta}_d (f_d)v_d ,v_0 \rangle
\\
&=\phi(f_1\otimes \cdots \otimes f_d)
\end{align*}
for all $f_1,\dots,f_d\in C_c(G)$. By totality and continuity, we get $\hat{\check{\phi}}= \phi$.
\end{proof}

\begin{proof}[Proof of Theorem \ref{thm2}] Hanh-Banach's theorem yields that $Y^\perp \cong (X_d'(G)/Y)^*$ isometrically. Thus, by Theorem \ref{thm1} we have an isometric isomorphism
\begin{align*}
i: \varphi\in M_d(G)\mapsto [\hat{\varphi}] \in ( X_d'(G)/Y)^*
\end{align*} such that
\begin{align}\label{eq dual}
[\hat{\varphi}] ([f_1\otimes \cdots \otimes f_d]) = (\varphi, f_1*\cdots *f_d)_{(L^\infty(G),L^1(G))}.
\end{align} for all $f_1,\dots,f_d\in L^1(G)$. If we restrict its dual map onto $X_d'(G)/Y$, we get an isometry
\begin{align*}
i^*: X_d'(G)/Y\rightarrow  M_d(G)^*
\end{align*} onto a closed subspace of $M_d(G)^*$ (which is a Banach predual of $M_d(G)$) such that $i^*([f_1\otimes\cdots \otimes f_d ]) = f_1*\cdots *f_d$ for all $f_1,\dots,f_d\in L^1(G)$. Since the inclusions $L^1(G)\otimes \cdots \otimes L^1(G)\subseteq X_d'(G)$ and $L^1(G)^{*d}\subseteq L^1(G)\subseteq X_d(G)$ are dense, we deduce that $i^*(X_d'(G)/Y) = X_d(G)$ and $M_d(G)\cong X_d(G)^*$ isometrically.
\end{proof}

\begin{proof}[Proof of Theorem \ref{thm AP}] The "if" part is handled in \cite[Theorem 1.4]{Bat23b}. We now prove the other direction.

Suppose that $\Omega\subseteq G$ is a measurable, finite measured fundamental domain for $\Gamma$ so that we have $G= \bigsqcup_{\gamma\in \Gamma} \gamma \Omega$. We can normalize the Haar measure $dw$ of $G$ so that $\Omega$ has measure 1. Define a unital mapping
\begin{align*}
i: \varphi\in L^\infty (G)\mapsto i (\varphi) \in \ell^\infty (\Gamma),\quad i(\varphi)(\gamma) = \int_{\Omega} \varphi(\gamma w)dw.
\end{align*}
We claim that this map is normal. To see that, it is enough to show that the dual map $i^*$ sends $\ell^1(\Gamma)$ into $L^1(G)$. Take any $f\in \ell^1(\Gamma)$.
Observe that
\begin{align*}
(\varphi, i^*(f))_{(L^\infty(G),L^\infty(G)^*)} 
&= (i(\varphi),f)_{(\ell^\infty(\Gamma),\ell^1(\Gamma))} = \sum_{\gamma\in \Gamma} f(\gamma) \int_\Omega \varphi (\gamma w)dw
\\&= \int_G \left ( \sum_{\gamma \in \Gamma} f(\gamma) 1_{\gamma \Omega}(g) \right ) \varphi (g)dg
\end{align*}
and
\begin{align*}
\int_G\left | \sum_{\gamma \in \Gamma} f(\gamma) 1_{\gamma \Omega}(w) \right | dw \leq \sum_{\gamma \in \Gamma} |f(\gamma)| |\gamma \Omega| = \|f\|_1 <\infty.
\end{align*} This shows that $i^* (f)$ is in $L^1(G)$ and the claim is true. 

Take any $\varphi \in M_d(G)$ and let $\xi_i$'s be as in Lemma \ref{lem cont}. Observe that
\begin{align*}
i(\varphi)(\gamma_1 \cdots \gamma_d)= \int_\Omega \varphi (\gamma_1 \cdots \gamma_d w)dw = \xi_1(\gamma_1)\cdots \xi_{d-1}(\gamma_{d-1}) \int_\Omega \xi_d(\gamma_d w)dw.
\end{align*}
This shows that the mapping
\begin{align*}
i:M_d(G) \rightarrow M_d(\Gamma)
\end{align*}
is well defined and norm decreasing. Moreover, it is $\sigma(M_d(G),X_d(G))$-$\sigma(M_d(\Gamma),X_d(\Gamma))$-continuous since $i^*:\ell^1(\Gamma)\rightarrow L^1(G)$ is well defined and the inclusions $\ell^1(\Gamma)\subseteq X_d(\Gamma)$ and $L^1(G) \subseteq X_d(G)$ are dense. 

Take any $\varphi\in C_c(G)$. Observe that 
\begin{align*}
\|i(\varphi)\|_2^2 =\sum_{\gamma\in \Gamma} \left | \int_\Omega\varphi(\gamma w) dw\right |^2 \leq \sum_{\gamma\in \Gamma}  \int_\Omega |\varphi(\gamma w)|^2 dw  \int_\Omega 1^2 dw = \|\varphi\|_2^2 <\infty.
\end{align*}
In other words, $i(\varphi) \in \ell^2(\Gamma)\subseteq A(\Gamma)$.

Now the proof is straightforward. As $G$ has $M_d$-AP, there is a net $(\varphi_i)$ in $C_c(G)\cap M_d(G)$ that converges to 1 in  $\sigma(M_d(G),X_d(G))$-topology. Then, $(i(\varphi_i))$ is a net in $A(\Gamma)$ that converges to 1 in  $\sigma(M_d(\Gamma),X_d(\Gamma))$-topology. As pointed out in \cite[Definition 2.2 and Remark 2.3]{Bat23b}, this is enough to deduce that $\Gamma$ has $M_d$-AP.
\end{proof}

\bibliographystyle{amsalpha}
\bibliography{mybibfile.bib}
\end{document}